\documentclass[onecolumn]{article}

\usepackage[margin = 1 in]{geometry}

\usepackage[utf8]{inputenc} 
\usepackage{url}            
\usepackage{booktabs}       
\usepackage{amsfonts}       
\usepackage{nicefrac}       
\usepackage{microtype}      
\usepackage{amsthm}
\usepackage{cite}

\usepackage{amsmath}
\usepackage{amssymb}
\usepackage{nccmath}
\usepackage{mathtools}
\usepackage{xcolor}
\usepackage{graphicx} 

\let\labelindent\relax

\usepackage{wrapfig}
\usepackage{enumitem}
\usepackage{subfig}
\usepackage{cleveref}

\newtheorem{theorem}{Theorem}

\newcommand{\tr}{\mathrm{tr}}
\newcommand{\Pa}{\mathrm{P_1}}
\newcommand{\Pb}{\mathrm{P_2}}

\renewcommand{\baselinestretch}{0.985}

\makeatletter

\DeclareMathOperator*{\argmin}{arg\,min}

\renewcommand{\Re}{\mathbb{R}}

\newcommand{\T}{\mathcal T}

\newcommand{\M}{\mathcal M}

\newcommand{\E}{\mathbb E}

\newcommand{\I}{\mathcal I}
\newcommand{\N}{\mathcal N}

\newcommand{\opu}{\operatorname{U}}

\newtheorem{theorem*}{Theorem}[section]

\newtheorem{proposition}{Proposition}[section]

\newtheorem{remark}{Remark}[section]

\title{Best Response Strategies for Asymmetric Sensing \\ in Linear-Quadratic Differential Games 
\thanks{Research of the authors was supported by the ARL grant ARL DCIST CRA W911NF-17-2-0181.
}
\thanks{
Shubham Aggarwal and Tamer Ba{\c s}ar are with the Coordinated Science Laboratory at the University of Illinois Urbana-Champaign, Urbana, IL, USA-61801. \\
Dipankar Maity is with Department of Electrical and Computer Engineering at the University of North Carolina at Charlotte, Charlotte, NC, USA-28223.\\
Emails:\texttt{\{sa57, basar1\}@illinois.edu, dmaity@charlotte.edu}
}

}

\author{Shubham~Aggarwal, Tamer~Ba{\c s}ar, and Dipankar~Maity 
}
\begin{document}

\date{}
\maketitle

\begin{abstract}
    In this paper, we revisit the two-player continuous-time infinite-horizon linear quadratic differential game problem, where one of the players can sample the state of the system only intermittently due to a sensing constraint while the other player can do so continuously. Under these asymmetric sensing limitations between the players, we analyze the optimal sensing and control strategies for the player at a disadvantage while the other player continues to play its security strategy. We derive an optimal sensor policy within the class of stationary randomized policies. Finally, using simulations, we show that the expected cost accrued by the first player approaches its security level as its sensing limitation is relaxed. 
\end{abstract}

\section{Introduction}

Emerging applications such as autonomous driving, multi-robot systems, smart-grid, smart and connected cities \cite{hang2020human,marden2018game,saad2012game,chi2021game} require game theoretic formulations involving multiple players with time-varying and heterogeneous sensing/communication constraints.
Understanding the effects of the latter on the players' objective functions is a pivotal research direction for efficient and optimal operations of such networked systems. 

Major research investigations have focussed on single-player networked control systems that constitute two decision-makers: the \textit{controller} and the \textit{sampler} \cite{zhang2019networked}. 
The controller decides on the control inputs based on the measurements received via the sampler, whose objective is to reduce the sensing/communication burden on the system. 
In this context, the \textit{event-triggered} and \textit{self-triggered} \cite{ramesh2013design,heemels2021event} paradigm is an important line of work which has also been extended to the case of \textit{cooperative} multi-agent systems \cite{dimarogonas2011distributed, yi2018dynamic}. 
In this paper, however, we are concerned with a \textit{non-cooperative multi-player game} formulation---a setting which has not received proportionate attention---under resource constraints.

Research on noncooperatve games with sensing/commu-nication constraints have mimicked the developments on \textit{event-} and \textit{self-triggered} frameworks and imposed similar pre-defined sensing/communication policies for the players; see e.g., \cite{yuan2017event}.
This line of work has not considered designing equilibrium sensing/communication strategies. 
From a theoretical side, considering both control and sensing/communication actions as part of the players' strategies is a challenging research direction as the information structure becomes decentralized and (\textit{sampler}'s) action dependent.
Even simpler games (such as two-player linear-quadratic (LQ) games) with well-known analytical (saddle-point) control strategies may become intractable under these additional sensing/communication constraints, and finding \textit{security strategies} thus becomes challenging.

It is worth noting that the single agent setup with a joint optimization on control and sensing/communication is itself an extremely challenging problem, see for example \cite{imer2005optimal,imer2005optimalb,imer2006optimal,molin2009lqg,imer2010optimal,lipsa2011remote,molin2014optimal}, where the authors solve different variations of the joint sensing/control problems under resource constraints.
The presence of another agent makes the problem significantly more challenging, often rendering a partially observed LQ game to be an infinite-dimensional optimization/game problem even with continuous observation \cite{bagchi1981linear}.

In this paper, we  revisit the classical two-player linear-quadratic zero-sum differential game (LQ-ZSDG) due to its wide applicability as well as  analytical tractability.  
We consider the set-up proposed in \cite{maity2023optimal} involving one of the players (minimizer) to be sensing limited.
This player cannot sense the state continuously and it has to intermittently turn on the sensor to obtain the state measurements. 
The sensor turns itself off immediately after making a measurement and sending it to the controller. 
A new sensing request is required to turn the sensor on again. 
Due to resource limitation, the player has a constraint on the average number of times it may turn on the sensor. 
The objective of this work is to analyze the differential game under this asymmetric sensing limitation and obtain: (i) the optimal controller and sensing strategies for the minimizing player, (ii) the effect of the sensing constraint on the minimizing player's cost function.

\textbf{Prior work on games with intermittent sensing:} One of the first papers that addresses zero-sum differential games with sampled state information for the players is \cite{bacsar1995minimax}, where however the sampling is fixed, but the dynamics switch modes intermittently.
The first work involving a joint controller and sampler for LQ-ZSDGs, on the other hand, appears to be \cite{maity2016strategies} where both players had access to only intermittent measurements. 
In that work, the minimizing (maximizing) player optimized an upper (lower) bound of the objective function in order to derive a guaranteed performance bound. 
The necessary condition on the equilibrium/optimal sensing strategy was discussed, and the sensing strategy was left open as an optimization problem. 
Later, in \cite{maity2016optimal} saddle-point control strategies were derived for the true objective function (as opposed to optimizing a bound on the objective function considered in \cite{maity2016strategies}) with the added constraint that, no matter which player requests the sensing, the state measurement will also be revealed/leaked to the other player. 
This resulted in a \textit{cooperative} sensing strategy and the sensing problem became a joint optimization (instead of a game) among the players. 
In a later extension \cite{maity2017linear}, each player was allowed to select its own sensing instances without revealing the state information to the other player. 
While the equilibrium control strategies were derived, derivation of the equilibrium sensing strategy turned out to be an intractable problem.
However, \cite{maity2017linear} discussed the existence of multiple equilibria and the corresponding necessary and sufficient conditions. 
Recently, this framework was adopted for asset defense applications \cite{huang2021defending}, which is a special case of the abovementioned setup where two players have two decoupled dynamics.  
Due to this decoupling, the sensing strategy does not involve a game, rather it becomes an optimization problem similar to what was proposed in \cite{maity2016strategies}. 
The authors proposed a bisection-search based algorithm to find the sensing instances; however, the algorithm lacks any optimality guarantee.
A similar setup with communication constraints was considered in \cite{maity2023efficient} which involved three agents: two players and a remote sensor. 
The first player was sensing limited and had to rely on the remote sensor to sense the opponent, whereas the second player could continuously sense the first player. 
The equilibrium control strategies and the optimal communication strategy between the first player and the remote sensor were derived.

Extension of sensing limited LQ games to a non-zero-sum setting was first studied for an infinite horizon case in \cite{maity2017asymptotic}.
Both average cost and discounted cost formulations were investigated.
Later, this was extended to discrete-time finite horizon non-zero-sum LQ games in \cite{maity2017linearcdc}.
Both the equilibrium control and sensing strategies were derived. 

\textbf{Contributions:} The contributions of this work are as follows: (i) We consider an infinite-horizon LQ-ZSDG setup between two players, one of which  (minimizer, called Player $\Pa$) is sensing limited. 
We derive the optimal control strategy as well as the optimal sensing strategy. 
(ii) We analytically quantify the degradation in Player $\Pa$'s performance as we vary the budget on the sensing constraint.
(iii) We prove that the optimal sensing follows a threshold-based policy on the \textit{Age-of-Information} at the controller. 
We discuss a tractable and analytical computational method for this threshold.
(iv) Although we consider an average cost formulation, our analysis converts the problem into a discounted cost formulation, and then recovers the solution to the average cost formulation by taking appropriate limits.
Hence, the analytical treatment of this work readily applies to a discounted cost formulation as well.

\textbf{Organization: } The rest of the paper is organized as follows. We formulate the two-player differential game problem in Section \ref{sec:Problem_form}. The optimal controller and sensor policies for $\mathrm{P}_1$ are computed in Sections \ref{sec:OptController} and \ref{sec:optSensor}, respectively.
We provide supporting numerical simulations in Section \ref{sec:numSims} and conclude the paper with its major highlights in Section \ref{sec:conc_disc}.

\textbf{Notations:} For a given time $t$, $t^-$ denotes the time right before $t$. We define the set $\mathbb N_0 := \{0,1,2, \cdots \}$. For symmetric matrices $X,Y$, the notation $X \succeq Y$ implies that $X-Y$ is positive semi-definite.  $\mathbb{I}[\cdot]$ denotes the indicator function.
$\N(\mu, \Sigma)$ denotes Gaussian distribution with mean $\mu$ and covariance $\Sigma$.

\section{Problem Formulation}\label{sec:Problem_form}

We consider an LQ-ZSDG between two players ($\Pa$ and $\Pb$) where one of the players (say $\Pa$) has a sensing limitation which constrains the sensing-rate. 
The scenario is motivated by pursuit-evasion games where players often rely on remote sensors (e.g., a radar) to detect and track their opponents; see \cite{maity2023efficient, maity2023optimal}. 
The remote sensor must be used sporadically as it is typically a shared resource that needs to serve other processes.

Let the state of the game evolve according to a stochastic linear differential equation
\begin{align} \label{eq:dyn}
\begin{split}
    dx(t) & = [Ax(t) +B_1u_1(t) + B_2u_2(t)] dt + G dW(t), \\
    x(0) & \sim \N(0,\Sigma_0),
\end{split}
\end{align}
where $x \in \Re^{n_x}$ is the state of the system, $u_i \in \Re^{n_u}_i$ is the control input of player $\mathrm{P}_i$, $ i=1,2$, and $\{W(t)\}_{t \geq 0}$ is a $p$--standard Brownian motion which is independent of the  initial state $x(0)$. 
We consider an infinite-horizon zero-sum game with
objective function  
\begin{align} \label{eq:payoff}
    J = \limsup_{T \rightarrow \infty} \frac{1}{T}\E \Big[ \int_0^T(\|x\|_Q^2 + \|u_1\|^2_{R_1} - \|u_2\|^2_{R_2}) dt\Big],
\end{align}
where $Q \succeq 0, R_i \succ 0$.
Player $\Pa$ is tasked to minimize $J$ whereas $\Pb$ maximizes it.
In order to avoid an ill-posed game (i.e., $\Pb$ can ensure $J\to \infty$ regardless of $\Pa$'s strategy), some parametric assumptions on $(A, B_1, B_2, Q, R_1, R_2)$ are needed.
A sufficient (and almost necessary) condition is that there exists a $P \succeq 0$ that satisfies the following generalized algebraic Riccati equation (GARE) 
\begin{align} \label{eq:ARE}
    A^{\!\top}\! P + PA + Q + P(B_2R_2^{-1}B_2^\top - B_1 R_1^{-1}B_1^\top)\! P = 0.
\end{align}
A detailed discussion on the necessary and sufficient conditions for the well-posedness of \eqref{eq:payoff} can be found in \cite{bacsar1998dynamic}.
Under the assumption that GARE \eqref{eq:ARE} has a positive-semidefinite solution, the LQ-ZSDG admits a unique saddle point in state-feedback form given by
\begin{align} \label{eq:Nash_u}
    u_i^*(t) = (-1)^i R_i^{-1}B_i^\top P x(t), 
\end{align}
where  $P$ is the minimal positive-semidefinite solution to \eqref{eq:ARE}.

This saddle-point strategy pair results in the expected cost: 
\begin{align}\label{eq:sec_level}
    J^* = \tr(PGG^\top).
\end{align}
The pair of saddle-point strategies \eqref{eq:Nash_u} provides  \textit{security strategies} \cite{bacsar1998dynamic} for both players with the unique security level being $J^*$ in \eqref{eq:sec_level}. 
That is, any player who deviates from \eqref{eq:Nash_u}, while the other uses \eqref{eq:Nash_u}, will receive a worse value than $J^*$. 

In this work, we consider $\Pa$ to be sensing-limited and not having continuous access to $x(t)$. 
An average sensing constraint is imposed on $\Pa$:
 \begin{align} \label{eq:sensingConstraint}
     \limsup_{T \rightarrow \infty} \frac{1}{T}\E [n(T)] \le b,
 \end{align}
 where $b$ is the sensing budget and $n(T)$ denotes the total number of sensing instances up to time $T$. 

Due to \eqref{eq:sensingConstraint}, $\Pa$ cannot implement \eqref{eq:Nash_u} and is \textit{forced} to deviate from the equilibrium strategy. 
However, this player can decide when it wants to sense $x(t)$, thus having some control over its sensing mechanism. We also assume that $\Pb$ does not know the budget $b$. The objective of this work thus is to determine the optimal sensing instances as well as the optimal control strategy to arrive at the minimum possible cost.

Player $\Pb$ can sense the state $x(t)$ continuously at every time instance $t$, whereas $\Pa$ can only access the state $x(t)$ \textit{intermittently} at discrete time instances while obeying the allowed constraints budget. 
In this regard, $\Pa$ designs a sensor-controller pair (denoted by $S$ and $C$, respectively), similar to a setup originally proposed in \cite{maity2016strategies}. 
We denote the combined policy of $\Pa$ by $\mu:= (\mu^{S}, \mu^{C})$, where $\mu^S$ and $\mu^C$ denote the sensing and the control policies, respectively.

Let the set of (possibly random) sensing instants for player $\Pa$ up to the current instant $t$ be denoted by $\T(t):= \{\tau_1, \cdots, \tau_{n(t)}\}$, where $\tau_k$'s are strictly increasing, $\tau_{{n}(t)} \leq t$ almost surely. 
The randomness in the sensing instances could be due to their dependence on the previously sensed $x$, which follows a stochastic process, or due to $\mu^S$ being a randomized strategy itself. 
$n(t)$ denotes the cardinality of $\T(t)$, i.e., the total number of sensing till $t$.
These sensing instances are not known to $\Pb$, and cannot be computed in advance since $\Pb$ does not know the budget $b$. 
$\Pb$ continues to implement \eqref{eq:Nash_u}, which is a \textit{security strategy} for $\Pb$, resulting in a payoff of at least $\tr(PGG^\top)$. We finally note here that once the strategy of $\Pb$ is fixed as above, the resulting game problem reduces to solving for the best response function of $\Pa$, which is what we consider in the rest of the work.

\begin{remark}
    While it might be tempting for $\Pb$ to deviate from \eqref{eq:Nash_u} to exploit $\Pa$'s sensing limitation, this may not be possible since $\Pb$ faces an incomplete information game. 
    Without knowing $b$, $\Pb$ cannot optimally respond to $\Pa$'s strategy. 
    In fact, by `wrongly' deviating from \eqref{eq:Nash_u}, $\Pb$ can be worse off than sticking with \eqref{eq:Nash_u}, i.e., it may receive a payoff which is less than $\tr(PGG^\top)$.
\end{remark}

\subsection{Information Set and Admissible Policies for $\Pa$}
We let $\I^S(t) := \{x(s), u_1(r), \T(t^-) \mid r \in [0,t), s \in \T(t^-)\} $ to be the information available to the decision-maker for sensing, whereas that for the control is 
 $\I^C(t) := \{x(s), u_1(r), \T(t) \mid r \in [0,t), s \in \T(t)\}$. 
Note that the slight difference in these two information sets lies in whether $\T(t) \supset \T(t^-)$, i.e., whether a sensing occurred exactly at $t$ or not. 
 For $t=0$, we define $\I^S(0) := \emptyset$ and $\I^C(0) := \T(0)$. 

We define the space of admissible controller policies for $\Pa$ as $\M^C:= \{\mu^C \mid \mu^C \text{ is adapted to } \I^C\}$ and that of the admissible sensor policies as $\M^S:= \{\mu^S \mid \mu^S \text{ is adapted to } \I^S\}$.

\section{Optimal Estimator-Controller for $\Pa$} \label{sec:OptController}

Given that $u_2$ follows \eqref{eq:Nash_u}, we may rewrite $J$ in \eqref{eq:payoff} as 
\begin{align} \label{eq:J_unsquared}
     J = \limsup_{T \rightarrow \infty} \frac{1}{T}\E \Big[ \int_0^T(\|x\|_{\tilde Q}^2 + \|u_1\|^2_{R_1} ) dt\Big],
\end{align}
where $\tilde Q = Q - P B_2 R_2^{-1} B_2^\top P$ is positive-semidefinite \cite{bacsar1998dynamic}.
Substituting $u_2 = R_2^{-1} B_2^\top P x$ in the dynamics \eqref{eq:dyn} yields
\begin{align} \label{eq:dyn2}
\begin{split}
    dx(t) & = [\tilde Ax(t) +B_1u_1(t)] dt + G dW(t), \\
    x(0) & \sim \N(0,\Sigma_0),
\end{split}
\end{align}
where $\tilde A = A + B_2 R_2^{-1} B_2^\top P$. 
For notational convenience, we define
    $M_i : = P B_i R_i^{-1} B_i^\top P,~~ i = 1,2$.
After the standard completion of squares, we may rewrite $J$ in \eqref{eq:J_unsquared} as 
\begin{align*}
    J \!= \!\tr(\tilde P GG^\top) + \limsup_{T \rightarrow \infty} \frac{1}{T}\E \Big[ \!\int_0^T \!\! \|u_1 \!+\! R_1^{-1}B_1^\top \tilde P x\|_{R_1}^2 dt \Big],
\end{align*}
where $\tilde P$ is the unique positive-semidefinite solution of the algebraic Riccati equation
\begin{align}
    \tilde A^\top \tilde P + \tilde P \tilde A + \tilde Q - \tilde P B_1 R_1^{-1}B_1 \tilde P = 0.
\end{align}
It is also known \cite{bacsar1998dynamic} that $\tilde P = P$, and consequently, 
\begin{align} \label{eq:J}
    J & \!= \!\tr( P GG^\top) \!+ \!\limsup_{T \rightarrow \infty} \frac{1}{T}\E \Big[ \!\int_0^T \!\! \|u_1 \! +\! R_1^{-1}B_1^\top  P x\|_{R_1}^2 dt \Big] \nonumber \\
    & = J^* + \limsup_{T \rightarrow \infty} \frac{1}{T}\E \Big[ \!\int_0^T \!\! \|u_1+R_1^{-1}B_1^\top  P x\|_{R_1}^2 dt \Big].
\end{align}
Thus, with sampled measurements, the optimal controller for $\Pa$ takes the form
\begin{align} \label{eq:optimalControl}
    u_1^*(t) : = \mu^C(\I^C(t)) = - R_1^{-1}B_1^\top P \hat{x}^*(t),
\end{align}
where  $\hat{x}^*(t) = \E[x(t) \mid \I^C(t)] $ is the least-squares estimate for $x(t)$ under the information available to the controller.

Let $\tau(t)$ denote the latest sensing instance at any given time $t$, i.e., $ \tau(t) = \tau_{n(t)}$. Then, from the dynamics \eqref{eq:dyn2}, we obtain 
\begin{align*}
    x(t) &= e^{\tilde A (t-\tau(t))} x(\tau(t)) + \int_{\tau(t)}^t  e^{\tilde A (t-s)} B_1 u_1(s) ds + \int_{\tau(t)}^t  e^{\tilde A (t-s)} G dW(s),
\end{align*}
which yields 
\begin{align*}
    \E[x(t)\! \mid \! \I^C(t)] = e^{\tilde A (t-\tau(t))} x(\tau(t)) + \!\!\int_{\tau(t)}^t\! \! e^{\tilde A (t-s)} B_1 u_1(s) ds,
\end{align*}
where we have used the properties of the Brownian motion to conclude $\E[  \int_{\tau(t)}^t  e^{\tilde A (t-s)} G dW(s) \mid \I^C(t)] = 0$. 
Therefore, we notice that the estimate $\hat{x}^*(t)$ follows the dynamics 
\begin{align} \label{eq:estimator}
\begin{split}
    &\dot{\hat{x}}^*(t)  = \tilde{A} \hat{x}^*(t) + B_1 u_1(t), ~~\quad  t \in [0,\infty) \setminus \T(T)\\
    &\hat{x}^*(t)  = x(t), \qquad\qquad\qquad~~~~ t \in \T(T).
    \end{split}
\end{align}

We note that the estimator \eqref{eq:estimator} itself does not depend on the controller strategy (i.e., the emergence of a \textit{separation principle}).
Under the optimal control strategy of \eqref{eq:optimalControl}, we may further simplify the estimator to obtain 
\begin{align} \label{eq:OptimalEstimator}
\begin{split}
    &\dot{\hat{x}}^*(t)  = (\tilde{A} - P^{-1}M_1) \hat{x}^*(t) , ~~\quad  t \in [0,\infty) \setminus \T(T)\\
    &\hat{x}^*(t)  = x(t), \qquad\qquad\qquad\qquad t \in \T(T).
    \end{split}
\end{align}
Let $e(t) = x(t) - \hat x^*(t)$ denote the estimation error.
Notice that $\hat x^*$ (or equivalently, $e(t)$) is completely characterized by $\T(T)$, which is determined by the sensing strategy $\mu^S$. 
After substituting the optimal controller for $\Pa$ in \eqref{eq:J}, we obtain 
\begin{align} \label{eq:OptJ}
    J=  J^* + \limsup_{T \rightarrow \infty} \frac{1}{T} \E [\int_0^T \|e\|_{M_1}^2  dt].
\end{align}
Next we turn towards constructing the optimal sensor policy.

\section{Optimal Sensor Policy}\label{sec:optSensor}
 The key ingredient facilitating the construction of an optimal sensing policy (as we will see later) will be the observation that we can work with an equivalent countable state Markov decision process (MDP) rather than an uncountable one. Consequently, we use the properties of the optimal value function to derive the threshold structure of the optimal policy, and then provide an algorithm to compute the threshold parameter explicitly.

\subsection{Sensor optimization problem}
The sensor's objective is to control the estimation error $e(t)$ satisfying the following differential equation:
\begin{align}\label{eq:error}
\begin{split}
    de(t) &= \tilde A e(t)dt + GdW(t), \qquad t \in [0,\infty) \setminus \T(T) \\
    e(t) &= 0, \hspace{3.3 cm} t \in \T(T).
\end{split}
\end{align}
This results in 
    $e(t) = \int_{\tau(t)}^t  e^{\tilde A (t-s)} G dW(s).$

Let us define the sensor action at time $t$ as: $\{0,1\} \ni \delta(t):= \mu^S(\I^S(t))$, where $\mu^S \in \M^S$ is an admissible sensor policy and $\delta(t) \!=\! 1$ denotes that a sensing occurred at time $t$.  

Next, we define the \emph{age of information} of the controller as $\Delta(t) = t-\tau(t)$, which denotes the time elapsed since the last sensing instant. 
Then, one may notice that $\Delta(t)$ follows the controlled MDP:

\begin{align}\label{eq:AoI}
    \dot{\Delta}(t) & = 1, \qquad\qquad t \in [0, \infty)\setminus \T(T), \nonumber \\
    \Delta(t) & = 0, \qquad\qquad t \in \T(T).
\end{align}
Consequently, by defining the (state transition) matrix $\Phi_{\tilde{A}}(t):= e^{\tilde A t}$, we may compute $\E[e(t)e^\top(t)]$ as 
\begin{align}\label{eq:error_cov}
    & \E[e(t)e^\top(t)] \nonumber \\
    &= \E\Big[\Big(\int_{\tau(t)}^t \!\! \!\!\Phi_{\tilde A}(t-s)G dW(s)\Big) \Big(\int_{\tau(t)}^t \!\! \!\! \Phi_{\tilde A}(t-s)G dW(s)\Big)^{\!\!\top}\Big] \nonumber \\
    & = \E\Big[\int_{0}^{\Delta(t)} \Phi_{\tilde A}(s)GG^\top \Phi_{\tilde A}^\top(s) ds\Big],
\end{align}
where the second equality follows from temporal independence of Brownian motion. 
Substituting \eqref{eq:error_cov} in \eqref{eq:OptJ} yields

\begin{align}\label{eq:cost_sensor}
    J & = J^* + \limsup_{T \rightarrow \infty} \frac{1}{T} \E\Big[\int_0^T \tr\Big(M_1 \int_{0}^{\Delta(t)} \!\!\Phi_{\tilde A}(s)GG^\top \Phi_{\tilde A}^\top(s) ds\Big) dt\Big].
\end{align}

Thus, we seek the best response strategy of the sensor to minimize $J$ subject to the dynamics \eqref{eq:AoI} and the sensing constraint \eqref{eq:sensingConstraint}.
That is, we wish to find the optimal sensing strategy by considering 
\begin{align}\label{eq:argmin_sensor}
    \argmin_{\mu^S \in \M^S} \bar J(\mu^S) \text{      subject to    \eqref{eq:sensingConstraint},}
\end{align}
where $\bar J (\mu^S) := \limsup_{T \rightarrow \infty} \frac{1}{T} \E\Big[\!\!\int_0^T \!\!\tr\Big(M_1 \!\!\int_{0}^{\Delta(t)}\!\!\! \Phi_{\tilde A}(s)GG^\top $ $ \Phi_{\tilde A}^\top(s) ds\Big) dt\Big]$.

So far we have considered the case that the sensing could be performed at any time $t$, which implies the AoI $\Delta(t)$ is a continuous variable. 
This results in a continuous state-space constrained MDP, which is analytically, and sometimes computationally, intractable. 
To facilitate a simpler analysis and a tractable solution to the problem defined in \eqref{eq:argmin_sensor}, we further assume that the sensor can only sense at predefined time instances $\{0,h,\ldots, kh,\ldots\}$. 
Later, we will consider how $h$ affects the solution to \eqref{eq:argmin_sensor}, as well as consider the limiting case of $h\to 0$. 

Let the discrete time index $k$ determine the continuous time $t=kh$, and let discrete AoI be defined as $\Delta_k : = \frac{\Delta(kh)}{h}$. 
Therefore,
\begin{align}\label{eq:discretized_AoI}
  \Delta_{k+1} =  \begin{cases}
        \Delta_k + 1, & \delta_{k+1} = 0, \\
        1, & \delta_{k+1} = 1.
    \end{cases}
\end{align}
where $\delta_k = \delta(kh)$.
Equation \eqref{eq:discretized_AoI} is the discrete time equivalent of \eqref{eq:AoI}.

For any $t=kh$, we may write 
\begin{align*}
    \int_{0}^{\Delta(t)}\!\!\! \!\!\Phi_{\tilde A}(s)GG^\top \Phi_{\tilde A}^\top(s) ds & =\!\! \sum_{i=0}^{\Delta_k-1}\!\! \int_{ih}^{(i+1)h}\!\!\! \!\!\Phi_{\tilde A}(s)GG^\top \Phi_{\tilde A}^\top(s) ds  \nonumber \\
    & = \!\!\sum_{i=0}^{\Delta_k-1} \Phi_{\tilde A}(ih) \tilde{G}^h  \Phi_{\tilde A}^\top(ih)\!,
\end{align*}
where we define the constant  $\tilde G^h := \int_{0}^{h} \Phi_{\tilde A}(s)GG^\top \Phi_{\tilde A}^\top(s) ds$.

Therefore, in a similar fashion, we may write 
\begin{align}\label{eq:discretized_J}
\begin{split}
    & \bar{J}^h(\mu^S) := \limsup_{m \rightarrow \infty} \frac{1}{m} \E\Big[\sum_{k = 0}^m \opu(\Delta_k, \delta_k) \Big] \\
    & \opu(\Delta_k,\delta_k) := \tr\Big(M_1 \sum_{i=0}^{\Delta_k-1} \Phi_{\tilde A}(ih) \tilde{G}^h  \Phi_{\tilde A}^\top(ih)\Big) ,
\end{split}    
\end{align}
where we have used the superscript $h$ in $\bar{J}^h$ to remind us that the cost depends on the choice of $h$. 
Notice that $\opu(\Delta_k,\delta_k)$ does not directly depend on $\delta_k$, and the dependence is through $\Delta_k$ which gets reset to $0$ via $\delta_k$. 

\subsection{Sensor best-response strategy}\label{subsec:Sensor_BR}

Now, we wish to solve the constrained MDP problem with the objective $\bar J^h(\mu^S)$ subject to the dynamics \eqref{eq:discretized_AoI} and the discrete time equivalent of the constraint \eqref{eq:sensingConstraint}.
To this end, we first note that $\Delta_k \in \mathbb{N}_0$, and $Q(\Delta_k+1, \Delta_k,\delta_k) = \delta_k$ and $Q(1, \Delta_k,\delta_k) = 1-\delta_k$, where $Q(s',s,a)$ denotes the probability of transitioning to the next state $s'$ given that an action $a$ was chosen in the current state $s$. Next, we define the Lagrangian function corresponding to \eqref{eq:discretized_J}  as:
\begin{align}\label{eq:unconstrained_J}
    \bar J^h_\lambda(\mu^S) & = \limsup_{m \rightarrow \infty} \frac{1}{m} \E\Big[\sum_{k = 0}^m \opu(\Delta_k, \delta_k) + \lambda n(m) \Big] \nonumber \\
    &  = \limsup_{m \rightarrow \infty} \frac{1}{m} \E\Big[\sum_{k = 0}^m \bar \opu(\Delta_k, \delta_k) \Big],
\end{align}
where $\bar \opu(\Delta_k, \delta_k) = \opu(\Delta_k, \delta_k) + \lambda \delta_k$ and $\lambda \ge 0$ is the Lagrange multiplier. Consequently, we obtain an unconstrained minimization problem defined by \eqref{eq:unconstrained_J} subject to \eqref{eq:discretized_AoI}, for which we obtain an optimal sensing policy in the sequel. The procedure \cite{aggarwal2023weighted} for the same involves first passing to a discounted cost problem corresponding to the average cost problem in \eqref{eq:unconstrained_J}, establishing the optimal policy structure for the former, and then taking sequential limit of the discount factor to 1 to recover the policy structure for the average cost problem \cite{aggarwal2023weighted}.

Let us start by considering a discounted cost formulation corresponding to \eqref{eq:unconstrained_J} with a discount factor of $\beta \in (0,1)$ as follows:
\vspace{-2mm}
\begin{align}
    \bar{J}^h_{\beta,\lambda}(\mu^S) & = \E\Big[\sum_{k=0}^\infty \beta^k \bar \opu(\Delta_k,\delta_k) \Big].
\end{align}
Further, let us also define the optimal value function associated with the above discounted cost by $V^h_{\beta,\lambda}(\Delta):= \inf_{\mu^S \in \M^S} \bar J^h_{\beta,\lambda}(\Delta,\mu^S)$. Then, it is easy to see that due to the positivity of the cost $\bar \opu(\cdot, \delta_k)$, $V^h_{\beta,\lambda}(\Delta)$ is non-decreasing in its argument.
Moreover, we can show using \cite{aggarwal2023weighted} that $V^h_{\beta,\lambda}(\Delta)$ satisfies the Bellman equation
    \begin{align}\label{eq:disc_bellman}
        V^h_{\beta,\lambda}(\Delta) = \min_{\delta \in \{0,1\}}\{\bar \opu(\Delta,\delta) + \beta \E[V^h_{\beta,\lambda}(\Delta')]\},
    \end{align}
    and that a stationary (deterministic) policy solving the RHS in \eqref{eq:disc_bellman} is optimal, which we refer to as $\beta$--optimal.
    
The following proposition now characterizes the structure of the $\beta$-optimal policy as a threshold-based policy.

\begin{proposition}\cite{aggarwal2023weighted}
        The $\beta$--optimal policy is of the form
        $\delta_{k} = \mathbb{I}[\Delta_k \geq \eta_{_{\beta, \lambda}}],$
    for some $\eta_{_{\beta, \lambda}} \ge 0$.
\end{proposition}

Consider now a sequence $\beta_\ell$ of discount factors converging to $\beta$, and define $\bar \eta_\lambda$ to be the limit of the $\eta(\beta_{\ell_p})$, which are $\beta_{\ell_p}$--optimal thresholds with $\beta_{\ell_p}$ a converging subsequence of $\beta_\ell$. The limit $\bar \eta_\lambda$ is guaranteed to exist by the compactness of $\{0,1\}^\infty$.
We then have the main result as follows.

\begin{theorem}\label{thm:avg_optimality}
    There exist $ \bar V^h_\lambda := \lim_{\beta \nearrow 1}(1-\beta)V^h_{\beta,\lambda}(\Delta), \forall \Delta \in \mathbb{N}_0$ and $f^h_\lambda(\Delta):=V^h_\lambda(\Delta) - V^h_\lambda(1)$ satisfying
    \begin{align}\label{eq:avg_Bellman}
        \Bar V^h_\lambda \! + \!f^h_\lambda(\Delta) \!=\! \min_{\delta}\{\bar \opu(\Delta,\delta)\! +\! \E[f^h_\lambda(\Delta')]\}.
    \end{align}
\noindent Moreover, the limiting policy $\mu^{S,*}$ generating actions according to  $\delta_k ^*= \mathbb{I}[\Delta_k \geq \bar \eta_\lambda]$ as constructed above is average cost optimal (i.e., it achieves the minimum on the RHS of \eqref{eq:avg_Bellman}).
\end{theorem}

\begin{proof}
    To prove the above, we first observe that $V^h_{\beta,\lambda}(\Delta) - V^h_{\beta,\lambda}(1) \geq 0$ using monotonicity of $V^h_{\beta,\lambda}(\cdot)$. Further, by using \eqref{eq:disc_bellman}, we have that $V^h_{\beta,\lambda}(\Delta) \leq \bar \opu(\Delta,1)+ \beta V^h_{\beta,\lambda}(1) \leq \bar \opu(\Delta,1)+ V^h_{\beta,\lambda}(1)$. This implies that $V^h_{\beta,\lambda}(\Delta) - V^h_{\beta,\lambda}(1) \leq \bar \opu(\Delta,1)$. 
    Next, let us fix an $i \in \mathbb N_0$. Then, we have that $\sum_{j \in \mathbb{N}_0} Q(j,i,1) \bar \opu(j,1) = \bar \opu(1,1) < \infty$ and $\sum_{j \in \mathbb N_0}Q(j,i,0) \bar \opu(j,0) = \bar \opu(i+1,0) < \infty$. Hence Assumptions 2 and $3^*$ in \cite{sennott1989average} are satisfied. The result then follows using the main theorem of \cite{sennott1989average}.
\end{proof}
The above theorem establishes the threshold structure of the optimal policy $\mu^{S,*}$. Moreover, the quantity $\bar V^h_\lambda$ is called the average cost, and is independent of the initial condition, and \eqref{eq:avg_Bellman} gives the Bellman equation corresponding to $V^h_\lambda(\Delta)$.

\subsection{Threshold computation \& Policy construction}\label{subsec:Random_policy}
We now proceed to computing the threshold parameter $\bar\eta_\lambda$, which will then completely characterize the optimal policy solving the unconstrained minimization problem defined by \eqref{eq:unconstrained_J} subject to \eqref{eq:discretized_AoI}. 
In this regard, let us first define the finite state space $\mathcal{K} = \{1, \cdots, \bar \eta_\lambda\}$. For notational convenience, we also relabel $\opu(\Delta,\delta)$ as $\opu(\Delta)$.
Then, for $\Delta = \bar \eta_\lambda$, we have from Theorem \ref{thm:avg_optimality} that $\delta^* = 1$. Similarly, for $\Delta = \bar \eta_\lambda -1$, we have $\delta^* = 0$. Thus, using \eqref{eq:avg_Bellman}, we arrive at $V^h_\lambda(\bar \eta_\lambda) \leq \lambda + V^h_\lambda(1) \le V^h_\lambda(\bar \eta^\lambda +1)$, from which we have that there exists $\theta \in [0,1]$ such that 
\begin{align}\label{eq:test0}
    V^h_\lambda(\bar \eta_\lambda + \theta) = \lambda + V^h_\lambda(1).
\end{align}
Further, for $\Delta \geq \bar \eta_\lambda$, we have that $\delta^* =1$, and thus
\begin{align}\label{eq:test1}
    V^h_\lambda(\Delta) = \opu(\Delta) + \lambda + V^h_\lambda(1) - \bar V^h_\lambda,
\end{align}
which upon simplification and using \eqref{eq:test0}, yields
\begin{align}\label{eq:test2}
    \bar V^h_\lambda = \opu(\bar \eta_\lambda + \theta).
\end{align}
Further, for $\Delta \le \bar \eta_\lambda$, we have that
    $V^h_\lambda(\Delta) + \bar V^h_\lambda = \opu(\Delta) + V^h_\lambda(\Delta+1)$,
which, further using \eqref{eq:test1} and \eqref{eq:test2}, yields
\begin{align}\label{eq:threshold_comp}
    \opu(\bar \eta_\lambda + \theta) \bar \eta_\lambda = \sum_{\ell =1}^{\bar \eta_\lambda} \opu(\ell) + \lambda.
\end{align}
The above is an implicit equation which can be solved numerically to find $\bar \eta_\lambda$.
Thus, we have completely characterized the solution to the unconstrained MDP  defined in \eqref{eq:unconstrained_J}.

Now, we return to our original constrained optimization problem of minimizing \eqref{eq:discretized_J} subject to the sensing constraint \eqref{eq:sensingConstraint}. To this end, we start by observing that the sensing rate for $\mathrm{P}_1$ is given as $1/\bar \eta_\lambda$, and thus the constraint \eqref{eq:sensingConstraint} can be rewritten as 
\vspace{-1mm}
\begin{align}\label{sensing_rate}
    \frac{1}{\bar \eta_\lambda} \leq bh.
\end{align}
 It is also easy to see that an optimal policy (if it exists) satisfies the above constraint with equality. However, for such constrained optimization problems, an optimal policy may not lie in the class of deterministic policies \cite{altman2021constrained}. Thus, we construct a randomized policy ($\mu^S_r$) for the same, similarly to \cite{chen2021minimizing,aggarwal2023weighted}. We briefly highlight the construction here.

We first compute the optimal Lagrange multiplier ($\lambda^*$) in \eqref{eq:unconstrained_J}. To do so, we use the \textit{iterative bisection search method}. Precisely, we initialize $\lambda_1^{(0)}=0$ and $\lambda^{(0)}_2=1$, and use the bisection method to find $\lambda^*$ within the interval $[\lambda_1^{(0)},\lambda^{(0)}_2]$ by satisfying \eqref{sensing_rate}. If it is not satisfied, we set $\lambda_1^{(i+1)} = \lambda_2^{(i)}$ and $\lambda_2^{(i+1)} = \lambda_2^{(i)}+1$, where $i$ is the iteration index.
We stop, whenever \eqref{sensing_rate} is satisfied, say, in the interval $[\lambda_1^{(i^*)}, \lambda_2^{(i^*)}]$ and $\lambda_2^{(i^*)} < \lambda_1^{(i^*)} + \epsilon$, for an appropriately chosen tolerance $\epsilon>0$. Next, we compute the thresholds $\bar \eta_{\lambda_1^{(i^*)}}$ and $\bar \eta_{\lambda_2^{(i^*)}}$ by rounding up to the nearest integer values. Let $b_1$ and $b_2$ be the budgets utilized by the policies corresponding to the thresholds $\bar \eta_{\lambda_1^{(i^*)}}$ and $\bar \eta_{\lambda_2^{(i^*)}}$, which are computed by solving \eqref{eq:threshold_comp} using $\lambda_1^{(i^*)}$ and $\lambda_2^{(i^*)}$, respectively. 
The discussion on the optimal sensing strategy can then be summarized in the following proposition.

\begin{proposition}
    The optimal sensing policy is randomized in general, and has the form 
    \begin{align}
        \delta_k^* = \mathbb I [\Delta_k \ge \vartheta \bar \eta_{\lambda_1^{(i^*)}} + (1-\vartheta)\bar \eta_{\lambda_2^{(i^*)}} ]
    \end{align}
    where $\vartheta\sim \mathrm{Bernoulli}(\nicefrac{(b-b_2)}{(b_1 - b_2)})$.
\end{proposition}

\section{Numerical simulations}\label{sec:numSims}
In this section, we validate through simulations the theoretical findings of the previous sections. For all simulations, we consider a scalar system (i.e., $n_x=1$ in \eqref{eq:dyn}), for which the system parameters are given as: $A = 0.5,~Q=4,~B_1=1,~B_2=0.5,~R_1=1,~R_2=0.5$. We first study the effect of the discretization carried out in Section \ref{subsec:Sensor_BR}. 
We plot the variation of the closed-loop cost $\bar J^h(\mu^S_r)$ versus the discretization parameter $h$ for a given budget of $0.4$ in Figure \ref{Fig:Cost_vs_h}. 
From that figure, we observe that a lower $h$ (mostly) leads to a lower cost  $\bar J^h(\mu^S_r)$. 
This is because we get a better approximation of the continuous time system with a lower $h$. Thus, by choosing $h$ judiciously, one can trade off computation resources required for sensing and the performance of the system.
\begin{figure}[h]
	\centering
	\includegraphics[width=0.59 \columnwidth ]{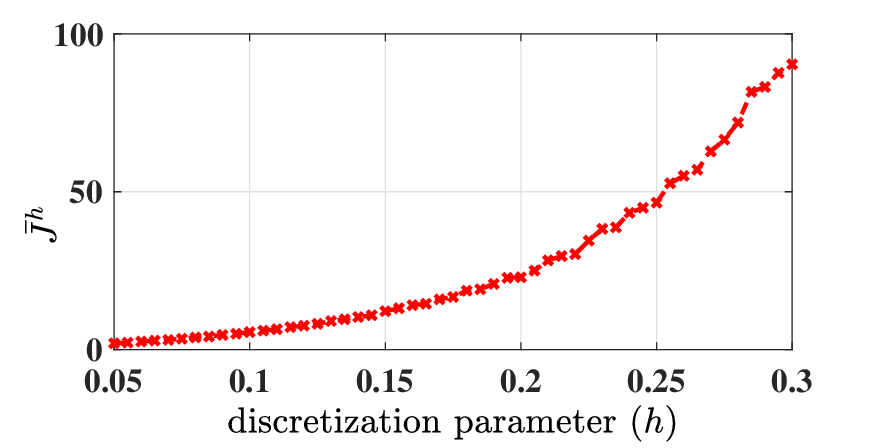}
 \vspace{-0.15cm}
	\caption{\small{Variation of the closed-loop cost $\bar J^h(\mu^S_r)$ vs $h$.
	}}
	\label{Fig:Cost_vs_h}
\end{figure}

Next, we plot (in Figure \ref{Fig:Cost_vs_C}) the variation of the closed-loop cost $\bar J^h(\mu^S)$ under the randomized policy $\mu^S_r$ versus the sensing budget $b$, for the discretization parameter $h=0.1$.
From this plot we observe that as more and more sensing budget $b$ is allowed to $\mathrm{P}_1$, its performance improves, as one would expect.
\begin{figure}[h]
	\centering
	\includegraphics[width= 0.59\columnwidth]{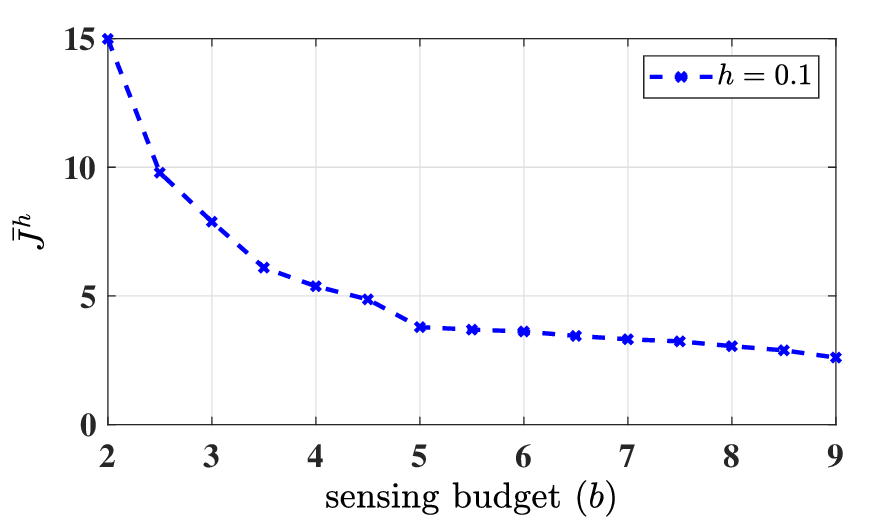}
 \vspace{-0.15cm}
	\caption{\small{Variation of the closed-loop cost $\bar J^h(\mu^S_r)$ vs $b$.
	}}
	\label{Fig:Cost_vs_C}
\end{figure}

Finally, we simulate a system of two players as formulated in the paper. The sensing budget was taken to be 40\% of the total horizon of 500 seconds and the discretization parameter was taken to be $h = 0.001$. The values of $\bar \eta_{\lambda_1^{(i^*)}}$ and $\bar \eta_{\lambda_2^{(i^*)}}$ were obtained to be 2 and 3 respectively. The corresponding evolution of the closed-loop state $x(t)$, estimation error $e(t)$ for player 1 and the control inputs of both the players versus time are plotted in Fig. \ref{Fig:evolution_vs_time}. 

\begin{figure*}[h]
	\centering
	\includegraphics[width=\textwidth]{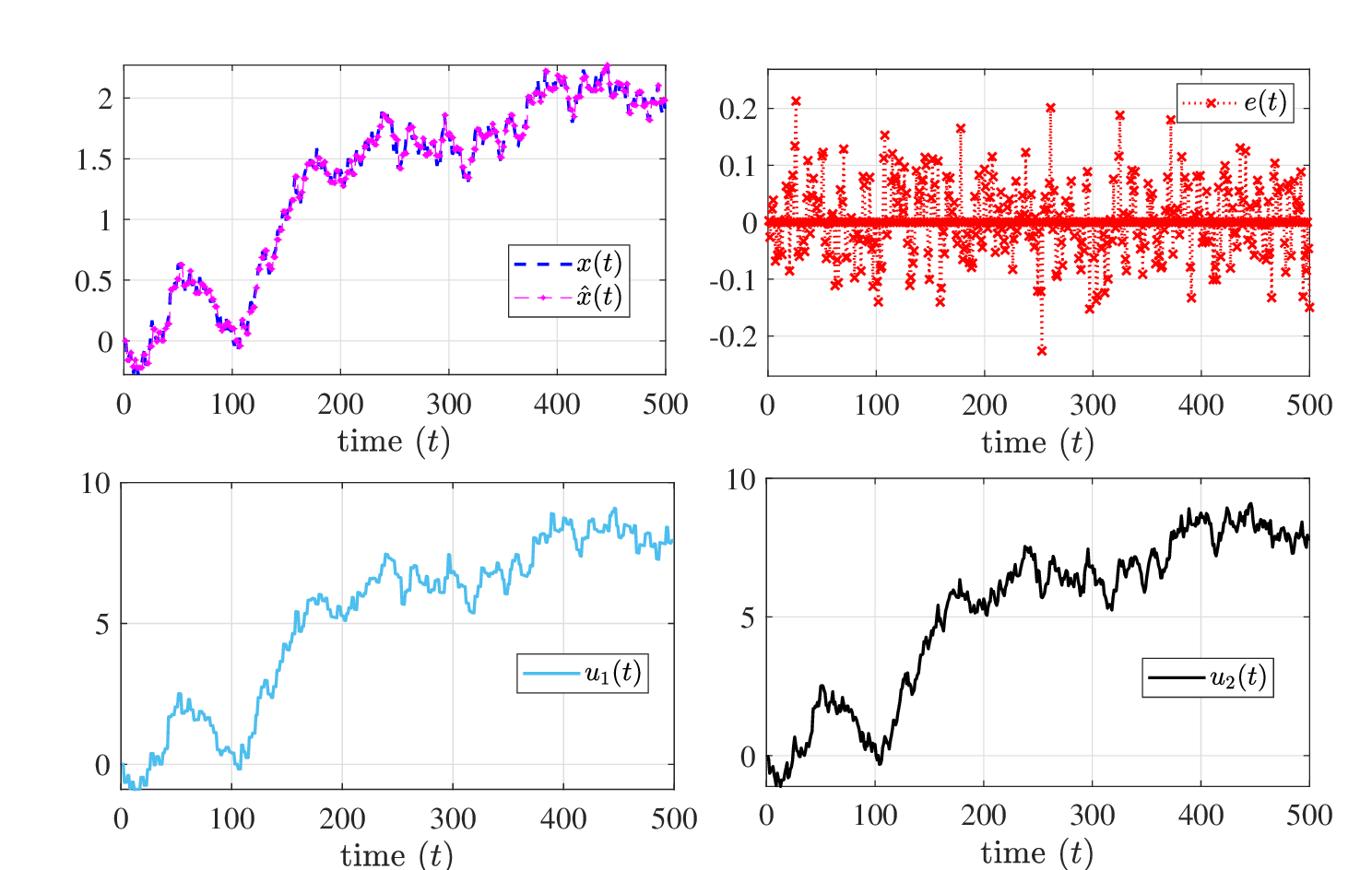}
 \vspace{-0.15cm}
	\caption{\small{Time evolution of the closed-loop state $x(t)$, estimation error $e(t)$ of player 1 and the control inputs $u_1(t),u_2(t)$ for both the players.}}
	\label{Fig:evolution_vs_time} 
\end{figure*}

\section{Conclusions}\label{sec:conc_disc}
In this paper, we considered a two-player LQ-ZSDG, where one of the players (maximizer) can continuously sense the state of the system to compute its control inputs, whereas the other player (minimizer) has a sensing budget, which forces it to sense only intermittently and maintain a state estimate between these instants.
We addressed the problem of joint control-sensor design for the minimizing player while the maximizing player is restricted to its saddle-point policy of the perfect information game.
We first showed that the control and sensor designs can be decoupled (in a particular order). Then, by converting the resulting optimization problem into an (approximate) countable state MDP, we were able to explicitly compute the sensing instants of the optimal (stationary) randomized sensing policy.
Finally, using numerical simulations, we observed that the performance of the minimizing player improves with decreasing step-size and increasing sensing budget, as expected.

We also note that it is still an open question to find (saddle-point) equilibrium strategies for the players, for not only the case of asymmetric information as considered here, but also for the case where there is a sensing budget for both players.

\bibliography{references,maity}
\bibliographystyle{IEEEtran}

\end{document}